\documentclass[12pt]{article} 
\usepackage[T2A]{fontenc} 
\usepackage[latin1]{inputenc} 
\usepackage{comment, amsthm, amssymb, amsmath, mathrsfs} 
\usepackage[english]{babel}
\newtheorem{theorem}{Theorem} 
\newtheorem{lemma}{Lemma}[section]
\newtheorem{cor}{Corollary}[section]

\theoremstyle{remark}
\newtheorem{rem}{Remark}

\begin{document}
\newcommand{\eps}{\varepsilon}
\newcommand{\var}{\operatorname{Var}}

\author{Dmitry Krachun,\thanks{The first author was partially supported by the NCCR SwissMAP of the Swiss National Science Foundation}\quad  Yuri Yakubovich\\St.Petersburg State University}
\title{On random partitions induced by random maps}
\maketitle

\begin{abstract}
The lattice of the set partitions of $[n]$ ordered by refinement is studied. Given a map $\phi: [n] \rightarrow [n]$, by taking preimages of elements we construct a partition of $[n]$. Suppose $t$ partitions $p_1,p_2,\dots,p_t$ are chosen independently according to the uniform measure on the set of mappings $[n]\rightarrow [n]$. The probability that the coarsest refinement of all $p_i$'s is the finest partitions $\{\{1\},\dots,\{n\}\}$ is shown to approach $1$ for any $t\geq 3$ and $e^{-1/2}$ for $t=2$. The probability that the finest coarsening of all $p_i$'s is the one-block partition is shown to approach 1 if $t(n)-\log{n}\rightarrow \infty$ and $0$ if $t(n)-\log{n}\rightarrow -\infty$. The size of the maximal block of the finest coarsening of all $p_i$'s for a fixed $t$ is also studied.

\end{abstract}
\section{Introduction}

For a given $n$ define $\Pi_n$ to be the set of all partitions of $[n]=\{1,2,\dots,n\}$ with partial order given by $p\preceq p'$ if every block of $p'$ is a union of blocks in $p$. This partially ordered set is known to be a lattice, see \cite{RSS}, that is, for any two partitions $p_1, p_2\in \Pi_n$ there exists the greatest lower bound $\inf\{p_1,p_2\}$ and the least upper bound $\sup\{p_1,p_2\}$. Namely, $\inf\{p_1,p_2\}$ is the partition given by all the non-empty intersections of blocks of $p_1$ and $p_2$, and $\sup\{p_1,p_2\}$ is the smallest partition whose blocks are union of those in both $p_1$ and~$p_2$. 

Every map $\phi:[n]\rightarrow [n]$ induces a partition $p_\phi$ of $[n]$ into non-empty preimages of $\phi$: $[n]=\cup_{i:\phi^{-1}(i)\not=\emptyset}\phi^{-1}(i)$. Throughout the paper we work with random partitions of $[n]$ chosen according to the uniform measure on the set of all mappings from $[n]$ to $[n]$. 

We study properties of $\inf_{1\leq i \leq t}p_i$ and $\sup_{1\leq i \leq t}p_i$ where $p_i$ are chosen independently. We shall be mostly interested in how likely $\inf_i p_i$ is to be the minimal partition $p_{\min}=\{\{1\},\dots,\{n\}\}$ and how likely $\sup_i p_i$ is to be the maximal partition $p_{\max}=\{[n]\}$. Similar questions for the case when partitions are taken according to the uniform measure on the set $\Pi_n$ were studied in great details in \cite{BP}, see also \cite{ERC}, and for different finite lattices with the uniform measure, see \cite{CW,JEAH}.

In order to keep notation more readable we avoid using integer part $\lfloor\cdot\rfloor$ when it is formally needed. So when an argument $a$ is supposed to be integer, say it represents a number of some objects or appears in bounds for summation or product, it should be understood as $\lfloor a\rfloor$. 

The rest of the paper is organized as follows. In Section~\ref{sec:infimum} we investigate the infimum of several random partitions; part of these results were
claimed by Pittel~\cite{BP} and we present a proof for the sake of completeness. Section~\ref{sec:Stirling} summarizes some known facts about the Stirling numbers of the second kind. Section~\ref{sec:supremum} deals with the supremum of random partitions.  In the last section we study the size of the maximal block of $\sup_{1\leq i \leq t}p_i$ for a fixed $t$.

\section{Infimum of several partitions}\label{sec:infimum}

In this section we study $\inf\{p_{\phi_1},\dots,p_{\phi_t}\}$ where $\phi_1,\dots,\phi_n$ are maps from $[n]$ to $[n]$ chosen independently. The threshold value for $t$ here turns out to be equal to $2$: if $t>2$ then the probability that $\inf_i p_{\phi_i}=p_{\min}$ tends to $1$ as $n$ tends to infinity, and for $t=2$ this probability tends to $e^{-1/2}$. Evidently, the first fact for $t>3$ would follow from this fact for $t=3$. We now formulate and prove these results.

\begin{theorem}
\label{infinum of three partitions}
Suppose three maps $\phi_1, \phi_2, \phi_3:[n]\rightarrow [n]$ are chosen independently according to the uniform measure on the set of all such maps. Then 
\[\lim_{n\rightarrow\infty}\mathbb{P}(\inf\{p_{\phi_1}, p_{\phi_2}, p_{\phi_3}\}=p_{\min})=1.\] 
\end{theorem}

\begin{proof}
We use a simple observation that if $p:=\inf\{p_{\phi_1}, p_{\phi_2}, p_{\phi_3}\}\not =p_{\min}$ then at least two elements in $p$ must be in the same block. Let $A$ be the set of all pairs $\{i, j\}$ such that $i$ and $j$ are in the same block in $p$. Then we have 
\[\mathbb{P}[\inf\{p_{\phi_1}, p_{\phi_2}, p_{\phi_3}\}\not =p_{\min}]\leq \mathbb{E}[|A|]=\binom{n}{2}\mathbb{P}[1\text{ and } 2\, \text{ are in the same block in } p].\]
This probability can be easily calculated explicitly and equals $\left( \mathbb{P}\left[\phi_1(i)=\phi_1(j)\right] \right)^3=n^{-3}$ which gives us
\[\mathbb{P}[\inf\{p_{\phi_1}, p_{\phi_2}, p_{\phi_3}\}=p_{\min}]=1-\mathbb{P}[\inf\{p_{\phi_1}, p_{\phi_2}, p_{\phi_3}\}\not =p_{\min}]\geq 1-\binom{n}{2}n^{-3}\rightarrow 1.\]
\vskip-1em
\end{proof}
The idea of the proof of the next theorem is given in \cite{BP}, we present it here for the sake of completeness. 
\begin{theorem}
\label{infimum of two partitions}
Suppose two maps $\phi_1, \phi_2:[n]\rightarrow [n]$ are chosen independently according to the uniform measure on the set of all such maps. Then 
\[\lim_{n\rightarrow\infty}\mathbb{P}[\inf\{p_{\phi_1}, p_{\phi_2}\}=p_{\min}]=e^{-1/2}.\] 
\end{theorem}

\begin{proof}
Here the argument is a bit more subtle. We denote $\inf\{p_{\phi_1}, p_{\phi_2}\}$ by $p$. Let $A$ be the set of two-element blocks in $p$. Let $B$ be the set of triples $\{i, j, k\}$ such that $i, j$ and $k$ are in the same block in $p$. We first note that 
\[\mathbb{E}\left[|B|\right]=\binom{n}{3}\mathbb{P}[1,\, 2 \text{ and } 3 \text{ are in the same block in } p]=\binom{n}{3}n^{-6}<n^{-3}.\]

Hence, with probability at least $1-n^{-3}$ the partition $p$ has blocks of sizes $1$ and $2$ only. We now study the random variable $|A|$ which counts the number of two-element blocks in $p$. In order to evaluate $\mathbb{P}(|A|=0)$ we first calculate factorial moments of $|A|$. For any fixed $k\geq 0$ we have
\begin{align*}
\mathbb{E}\left[\binom{|A|}{k}\right] &=\frac{1}{k!}\prod_{s=0}^{k-1}\binom{n-2s}{2} \cdot \mathbb{P}[\{1, 2\},\dots \{2k-1, 2k\} \text{ are blocks in } p] \\ & = \left(1+O\left(\frac{1}{n}\right)\right)\cdot\frac{n^{2k}}{2^k \cdot k!}\cdot\frac{1}{n^{2k}}=\frac{1+o(1)}{2^k\cdot k!},\qquad n\to\infty,
 \end{align*}
where $o(1)$ is uniform in $k$. Now it is easy to see that
\[\mathbb{P}\left[|A|=0\right]=\sum_{k=0}^\infty (-1)^k\cdot\mathbb{E}\left[\binom{|A|}{k}\right]=(1+o(1))\cdot\sum_{k=0}^\infty \frac{(-1)^k}{2^k\cdot k!}\rightarrow e^{-1/2},\qquad n\to\infty.\]
\vskip-1em
\end{proof}

\section{Some properties of the Stirling numbers of the second kind}\label{sec:Stirling}

  In order to estimate the probability that supremum of several random partitions is equal to $p_{\max}$ we shall need the notion of the Stirling numbers of the second kind. Recall that the Stirling number of the second kind $S(n, k)$ counts the number of ways to partition the set $[n]$ into $k$ blocks. It is clear that the number of surjective maps from $[k]$ to $[l]$ equals $S(k,l)\cdot l!$ as each such map gives rise to a partition of $[k]$ into $l$ blocks. We shall frequently use this fact in our calculations. The following well-known fact  and its corollary shall not be used in our arguments though it is useful to keep them in mind.  
  
\begin{lemma}[{{\cite[Theorem 3.2]{EHL}}}]
\label{log-concavity of Stirling numbers}
For a given $n$, the Stirling numbers of the second kind, $S(n,k)$ form a log-concave sequence in $k$. That is, for any $k=2\dots n-1$ we have 
\[S(n, k)^2\geq S(n, k-1)\cdot S(n, k+1).\]
\end{lemma}
\begin{cor}
For any natural number $n$ the quantity $\frac{S(n, k-1)}{S(n, k)}$ increases in $k$.  
\end{cor}

For the proof of the next lemma concerning the Stirling numbers of the second kind, we need the so-called multi-valued map principle.

\smallskip\noindent\textbf{Multi-valued map principle.} Let $f$ be a multi-valued map from a finite set $S$ to a finite set $T$. For $t\in T$ write $f^{-1}(t):=\{s\in S: t\in f(s)\}$. Then 
\[\frac{|S|}{|T|}\leq\frac{\max_{t\in T}|f^{-1}(t)|}{\min_{s\in S}|f(s)|}.\]

\begin{lemma}
\label{quotient of stirling numbers}
For any natural numbers $l\leq k$ the following inequality holds:
\[\frac{S(k, l-1)}{S(k, l)}\leq \frac{l(l-1)}{2(k-l+1)}.\]
\end{lemma}
\begin{proof}
Let $\mathbb{A}_k^l$ denote the set of all partitions of $[k]$ into $l$ blocks. Consider a multi-valued map $\tau:\mathbb{A}_k^l\rightarrow \mathbb{A}_k^{l-1}$ which takes a partition $p\in \mathbb{A}_k^l$ and glues any two of its blocks. It is clear that every element in $\mathbb{A}_k^l$ has $\binom{l}{2}$ images. Now, suppose a partition $p\in \mathbb{A}_k^{l-1}$ has blocks of sizes $x_1,x_2,\dots,x_{l-1}$. Then $|f^{-1}(p)|$ is equal to $\sum_{s=1}^{l-1}(2^{x_s-1}-1)$. Indeed, $|f^{-1}(p)|$ is simply the number of ways to split one of the blocks of $p$ into two, and the number of ways to split a block of size $x$ into two is given by $2^{x-1}-1$. Now note that $\sum_{s=1}^{l-1}(2^{x_s-1}-1)\geq \sum_{s=1}^{l-1}(x_s-1)=k-l+1$, thus the multi-valued map principle gives us 
\[\frac{S(k, l-1)}{S(k, l)}=\frac{|\mathbb{A}_k^{l-1}|}{|\mathbb{A}_k^l|}\leq \frac{l(l-1)}{2(k-l+1)}.\]
\vskip-2em
\end{proof}
\begin{rem}
Note that the inequality is asymptotically tight for $l=o(\sqrt{k})$, see \cite{LCH}.
\end{rem}
We shall sometimes need a weaker bound given by the following trivial corollary.
\begin{cor}
\label{quotient of stirling numbers, rough bound}
For any natural numbers $l\leq k$ the following inequality is valid:
\[\frac{S(k, l-1)}{S(k, l)}\leq \frac{k^2}{2}.\]
\end{cor}
In the proof of Theorem \ref{large block for t=3} we use the following lemma, see \cite[Corollary 5]{EGT}.
\begin{lemma}
Suppose we have sequences $k_i, n_i$. In the following we omit indexes to lighten the notation. Assume that $k/n=c+o(1)$, with $c\in (0,1)$. Then the following asymptotics for $S(n,k)$ holds: 
\[S(n, k)=n^{n-k}\cdot e^{g(c)\cdot n+o(n)},\]
where $g(c)=c+\log{\gamma}+(\gamma-c)\cdot\log{(\gamma-c)}-\gamma\cdot\log{\gamma}$ and $\gamma$ is the unique solution of $\gamma\cdot(1-e^{-1/\gamma})=c$.
\label{tight asymptotic for stirling numbers}
\end{lemma}
\begin{rem}
In \cite{EGT} much tighter asymptotic expansion is given for $k=cn+o(n^{2/3})$. In order to pass to the case $k=cn+o(n)$ we can use Lemma \ref{quotient of stirling numbers}.
\end{rem}

\section{Supremum of several partitions}\label{sec:supremum}

We now turn to studying the supremum of several randomly chosen partitions. Suppose maps $\phi_1, \phi_2,\dots, \phi_t:[n]\rightarrow [n]$ are chosen independently according to the uniform measure on the set of all maps. We are interested in the question of how likely $p:=\sup_{1\leq i \leq t}p_{\phi_i}$ is to be equal to $p_{\max}=\{[n]\}$. Here the threshold value of $t$ equals $\log(n)$ where $\log$ denotes the natural logarithm. That is, if $t=t(n)=\log(n)-w(n)$ with $w(n)\rightarrow\infty$ arbitrarily slowly then $\mathbb{P}[p=p_{\max}]\rightarrow 0$; whereas if $t=t(n)=\log(n)+w(n)$ then $\mathbb{P}[p=p_{\max}]\rightarrow 1$.  We start with the following technical result which shall be used several times.

\begin{lemma}
\label{moments of M}
Let $M$ be the number of one-element blocks in $\sup\{p_{\phi_1},\dots,p_{\phi_t}\}$.  Then 
\[\mathbb{E}[M]=ne^{-t}+O(te^{-t}),\qquad \var(M)=O\left(\max\bigl\{nte^{-2t},te^{-t}\bigr\}\right),\qquad n\to\infty,\]
where both $O(\cdot)$ are uniform in $t=1,\dots,2\log{n}$.
\end{lemma}

\begin{proof}
By the linearity of the expectation and due to the symmetry,
\begin{align*}
\mathbb{E}[M]&=n\cdot\mathbb{P}[\{1\}\text{ forms a one-element block in } p]  = n\cdot\left(\left(1-\frac{1}{n}\right)^{n-1}\right)^t.
\end{align*}
So
\[\mathbb{E}[M]-ne^{-t}=n\left[\left(1-\frac{1}{n}\right)^{n-1}-e^{-1}\right]\sum_{s=0}^{t-1}\left(1-\frac{1}{n}\right)^{s(n-1)}e^{s-t+1}.\]
As $n\to\infty$, the expression in brackets is of order $O(1/n)$ while each summand is bounded above by $e^{-t+1}\left(1-\tfrac{1}{n}\right)^{-s}\sim e^{-t+1}$ for $s<t\le2\log{n}$. Similarly,
\begin{align*}
\mathbb{E}\left[\binom{M}{2}\right]&=\binom{n}{2}\cdot\mathbb{P}[\{1\}\text{ and } \{2\} \text{ are two one-element blocks in } p] \\ & = \binom{n}{2}\cdot\left(\left(1-\frac{1}{n}\right)\left(1-\frac{2}{n}\right)^{n-2}\right)^t=\binom{n}{2}\cdot e^{-2t}+O(nte^{-2t}),\qquad n\to\infty. 
\end{align*}
Hence
\begin{align*}
\var(M)
=2\cdot\mathbb{E}\left[\binom{M}{2}\right]+\mathbb{E}[M]-(\mathbb{E}[M])^2=
O\left(\max\bigl\{nte^{-2t},te^{-t}\bigr\}\right).
\end{align*}
\vskip-2em
\end{proof}

Now we are ready to formulate and prove the result for the case $t-\log(n)\to-\infty$.

\begin{theorem}
\label{supremum is not p_max for t<log(n)}
Let $w:\mathbb{N}\rightarrow\mathbb{R}$ be a function such that $\lim_{n\rightarrow\infty}w(n)=\infty$ and $w(n)<\log(n)$. Let $t=t(n)=\log(n)-w(n)$ be an integer, and suppose maps $\phi_1, \phi_2,\dots,\phi_t:[n]\rightarrow [n]$ are chosen independently according to the uniform measure on the set of all maps. Then 
\[\lim_{n\rightarrow\infty}\mathbb{P}[\sup\{p_{\phi_1},\dots,p_{\phi_t}\}=p_{\max}]=0.\] 
\end{theorem}

\begin{proof}
  We denote $\sup\{p_{\phi_1},\dots,p_{\phi_t}\}$ by $p$. Let $M$ be the number of one-element blocks in $p$. We want to show that $\mathbb{P}[M=0]$ tends to zero as $n$ tends to infinity. In order to do this we plug $t=\log{n}-w(n)$ into the expressions of Lemma~\ref{moments of M} to find out that  
\begin{align*}
\mathbb{E}[M]=e^{w(n)}\left(1+O\left(\frac{\log{n}}{n}\right)\right),\qquad \var(M)=O\left(\frac{\log{n}}{n}e^{2w(n)}\right),\qquad n\to\infty.
\end{align*}
So we can use the Chebyshev inequality to bound the probability that $M$ equals zero:
\[\mathbb{P}[M=0]\leq\frac{\var(M)}{(\mathbb{E}[M])^2}=O\left(\frac{\log{n}}{n}\right)+\left(1+O\left(\frac{\log{n}}{n}\right)\right)\cdot e^{-w(n)}\rightarrow 0,\qquad n\to\infty.\]
\vskip-1em
\end{proof}

In order to prove that for $t=\log(n)+w(n)$ the partition $p:=\sup\{p_{\phi_1},\dots,p_{\phi_t}\}$ is likely to be equal to $p_{\max}$ we need the following three lemmas. The first lemma claims that blocks of size less than $c\cdot\sqrt{n}$ are unlikely to appear in $p$; the second lemma claims that blocks of size between $c\cdot\sqrt{n}$ and $\log{n}\cdot\sqrt{n}$ are also unlikely to appear in $p$. Finally, the third lemma claims that $p$ is unlikely to have two blocks of size at least $\log{n}\cdot\sqrt{n}$. 

\begin{lemma}
\label{no blocks of size o(sqrt(n))}
There exist an absolute constant $c > 0$ such that the following holds. Let $w:\mathbb{N}\rightarrow\mathbb{R}$ be a function such that $\lim_{n\rightarrow\infty}w(n)=\infty$ and $w(n)<\log(n)$. Let $t=t(n)=\log(n)+w(n)$ be an integer, and suppose maps $\phi_1, \phi_2,\dots,\phi_t:[n]\rightarrow [n]$ are chosen independently according to the uniform measure on the set of all maps. Let $p:=\sup\{p_{\phi_1},\dots,p_{\phi_t}\}$, then 
\[\mathbb{P}\left[p \text{ has a block of size at most } c\cdot \sqrt{n}\right]<9\cdot e^{-w(n)}.\] 
\end{lemma}

\begin{proof}
We may assume that $n$ is large enough for our argument to work. We fix $k\leq c\cdot \sqrt{n}$ with small enough $c$, say $c=\frac{1}{100}$, and bound the probability that $p$ has a block of size $k$: 
\begin{align*}
\mathbb{P}[p \text{ has a block of size } k] & \leq \binom{n}{k}\cdot\mathbb{P}[\{1,2,\dots,k\} \text{ is a block of } p] \\ & \leq 
\binom{n}{k}\left(\mathbb{P}[\phi(a)\not = \phi(b) \text{ for any } a\leq k < b]\right)^t. 
\end{align*}
Note that for a fixed $l$, the number of maps $\phi$ such that $\phi(a)\not = \phi(b) \text{ for any } a\leq k < b$ and the image of $\{1,2,\dots,k\}$ under $\phi$ has $l$ elements, equals $\binom{n}{l}\cdot l!\cdot S(k,l)\cdot(n-l)^{n-k}$. Thus this expression can be rewritten in terms of Stirling numbers of the second kind as follows: 
\begin{equation}
\label{probability of having a block of size k}
\begin{aligned}
\mathbb{P}[p \text{ has a block of size } k] & \leq
\binom{n}{k}\left(\sum_{l=1}^k \frac{1}{n^n}\cdot\binom{n}{l}\cdot l!\cdot S(k,l)\cdot(n-l)^{n-k}\right)^t \\ & \leq
\binom{n}{k}\left(\sum_{l=1}^k n^{l-n}\cdot S(k,l)\cdot(n-l)^{n-k}\right)^t.
\end{aligned}
\end{equation}
We now estimate the sum in parentheses. Denoting $n^{l-n}\cdot S(k,l)\cdot(n-l)^{n-k}$ by $f_k(l)$, we have, for $k\leq c\cdot\sqrt{n}$ 
\[f_k(k)=\left(1-\frac{k}{n}\right)^{n-k}\leq e^{-k}\left(1-\frac{k}{n}\right)^{-k}\leq e^{-k}\left(1+\frac{2k^2}{n}\right).\]
Now we want to show that as $l$ decreases from $k$ to $1$, $f_k(l)$ decreases fast enough. Namely, we have, for $2\leq l\leq k$
\begin{align}
\frac{f_k(l-1)}{f_k(l)}=\frac{1}{n}\cdot \frac{S(k, l-1)}{S(k, l)}\cdot\left(1+\frac{1}{n-l}\right)^{n-k}\leq \frac{e\cdot k^2}{2n}.
\label{bound for f_k(l-1)/f_k(l)}
\end{align} 
Here the last inequality is due to Corollary \ref{quotient of stirling numbers, rough bound}. Putting this together we obtain 
\begin{align*}
\mathbb{P}[p \text{ has a block of size } k] & \leq \binom{n}{k}\left(\sum_{l=1}^k n^{l-n}\cdot S(k,l)\cdot(n-l)^{n-k}\right)^t \\ & \leq
\binom{n}{k}\left(\sum_{s=0}^{k-1}e^{-k}\left(1+\frac{2k^2}{n}\right)\cdot\left(\frac{e\cdot k^2}{2n}\right)^s\right)^t \\ & \leq
\frac{n^k}{k!}\left(e^{-k}\left(1+\frac{5k^2}{n}\right)\right)^t \\ & \leq 
e^{-k\cdot w(n)}\cdot \frac{1}{k!}\cdot\left(1+\frac{5k^2}{n}\right)^{2\log{n}}\leq \frac{5}{k^2}\cdot e^{-w(n)}.
\end{align*}
The last inequality is valid for all sufficiently large $n$ and any $k\leq c\cdot\sqrt{n}$. Indeed, if $k<n^{1/4}$, we argue that $\left(1+\frac{5k^2}{n}\right)^{2\log{n}}\leq 2$ for sufficiently large $n$ and the inequality follows immediately. Otherwise, if $n^{1/4}\le k\le c\cdot\sqrt{n}$, for sufficiently large $n$ we have  $\frac{1}{k!}\cdot\left(1+\frac{5k^2}{n}\right)^{2\log{n}}\leq \frac{1}{k!}\cdot 2^{2\log{n}}<1$.

Summing over all possible $k\leq c\cdot\sqrt{n}$ we deduce that 
\[\mathbb{P}\left[p \text{ has a block of size at most } c\cdot \sqrt{n}\right]<9\cdot e^{-w(n)}.\] 
\vskip-1.5em
\end{proof}

\begin{lemma}
\label{no blocks of size less than log(n)sqrt(n)}
For any constant $c>0$ there exists an absolute constant $C > 0$ such that the following holds. Let $w:\mathbb{N}\rightarrow\mathbb{R}$ be a function such that\/ $\lim_{n\rightarrow\infty}w(n)=\infty$ and $w(n)<\log(n)$. Let $t=t(n)=\log(n)+w(n)$ be an integer, and suppose maps $\phi_1, \phi_2,\dots,\phi_t:[n]\rightarrow [n]$ are chosen independently according to the uniform measure on the set of all maps. Let $p:=\sup\{p_{\phi_1},\dots,p_{\phi_t}\}$, then 
\[\mathbb{P}\left[p \text{ has a block of size at least } c\cdot \sqrt{n} \text{ and at most } \log{n}\cdot\sqrt{n}\right]<C\cdot e^{-n^{1/2}}.\] 
\end{lemma}
\begin{proof}

We may assume that $n$ is large enough. Let us fix $k$ between $c\cdot \sqrt{n}$ and $\log(n)\cdot\sqrt{n}$ and estimate the probability that $p$ has a block of size $k$.
Similarly to (\ref{probability of having a block of size k}) we have  
\[\mathbb{P}[p \text{ has a block of size } k] \leq \binom{n}{k}\left(\sum_{l=1}^k n^{l-n}\cdot S(k,l)\cdot(n-l)^{n-k}\right)^t.\]

We now estimate the sum in parentheses, though this time slightly differently. We again denote $n^{l-n}\cdot S(k,l)\cdot(n-l)^{n-k}$ by $f_k(l)$. We have 
\[f_k(k)=\left(1-\frac{k}{n}\right)^{n-k}\leq e^{-k}\cdot\left(1-\frac{k}{n}\right)^{-k}\leq e^{-k}\cdot e^{2k^2/n}.\]
Now we want to show that as $l$ decreases from $k$ to $1$, $f_k(l)$ does not increase for too long. Namely, we have, for $2\leq l\leq k$
\begin{align*}
\frac{f_k(l-1)}{f_k(l)}=\frac{1}{n}\cdot \frac{S(k, l-1)}{S(k, l)}\cdot\left(1+\frac{1}{n-l}\right)^{n-k}\leq \frac{e}{n}\cdot\frac{S(k, l-1)}{S(k, l)}\leq\frac{e\cdot k^2}{2n\cdot(k-l+1)},
\end{align*}
where the last inequality is due to Lemma \ref{quotient of stirling numbers}. It is clear from this inequality that $f_k(l-1)\leq f_k(l)$ for all $l\leq k-\frac{2k^2}{n}$. Thus, $\max_l \{f_k(l)\}$ is bounded by $f_k(k)\cdot \left(\frac{e\cdot s}{2}\right)^{2s}$, where $s:=\frac{k^2}{n}\leq (\log{n})^2$.
We now use this bound to obtain
\begin{align*}
\mathbb{P}[p \text{ has a block of size } k] & \leq \binom{n}{k}\left(\sum_{l=1}^k n^{l-n}\cdot S(k,l)\cdot(n-l)^{n-k}\right)^t \\ & =
\binom{n}{k}\left(\sum_{l=1}^k f_k(l) \right)^t \leq \binom{n}{k}\left(k\cdot \max_l\{f_k(l)\} \right)^t \\ & \leq
\frac{n^k}{k!}\left(e^{-k+2k^2/n}\cdot k \cdot (2s)^{2s}\right)^t \\ & =
\frac{1}{k!}\left(e^{2s} \cdot k\cdot (2s)^{2s}\right)^t\cdot \frac{n^k}{e^{kt}} \\ & \leq 
\frac{1}{e^{k\log{k}/2}}\left(C_1\cdot n\cdot e^{s^2}\right)^t \\ & \leq
\frac{1}{e^{c\sqrt{n}\log{n}/2}}\left(C_1\cdot n\cdot e^{(\log{n})^4}\right)^{2\log{n}} \\ & \leq
e^{-2\sqrt{n}+4\cdot (\log{n})^5}\\ & \leq
\frac{1}{n}\cdot e^{-n^{1/2}}.
\end{align*}
The above estimate is valid for sufficiently large $n$ and some absolute constant $C_1$. Note that we used the fact that $\frac{(2se)^{2s}}{e^{s^2}}$ is bounded on $\mathbb{R}_{+}$. Finally, summing over all possible $k\leq \log{n}\cdot\sqrt{n}$ we deduce that for some constant $C$
\[\mathbb{P}\left[p \text{ has a block of size at least } c\cdot \sqrt{n} \text { and at most } \log{n}\cdot\sqrt{n}\right]<C\cdot e^{-n^{1/2}}.\] 
\vskip-1em
\end{proof}

\begin{lemma}
\label{large blocks merge}
Let $p'$ be a fixed partition of $[n]$ with all blocks of size at least $\log{n}\cdot\sqrt{n}$. If the map $\phi:[n]\rightarrow [n]$ is chosen randomly according to the uniform measure on the set of all maps, then for $n$ large enough
\[\mathbb{P}[\sup\{p_{\phi},p'\}\not = p_{\max}]<n^2\cdot e^{-(\log{n})^2/2}.\] 
\end{lemma}
\begin{proof}
Note that if $\sup\{p_{\phi},p'\}\not = p_{\max}$ then in $p'$ there exist two blocks $\{x_1,x_2,\dots,x_a\}$ and $\{y_1, y_2,\dots,y_b\}$ such that $p_{\phi}$ does not `merge' these two blocks, that is, $\phi(x_i)\not = \phi(y_j)$ for any $i\leq a$ and $j\leq b$. We now want to show that the probability of such event is small for any two fixed blocks. It is sufficient to consider the case when both blocks have size $t=\log{n}\cdot\sqrt{n}$. Since the number of blocks in $p'$ is at most $\sqrt{n}$, the union bound gives us
\begin{align}
\mathbb{P}[\sup\{p_{\phi},p'\}\not = p_{\max}]\leq n\cdot\mathbb{P}[\phi(\{1,2,\dots, t\})\cap\phi(\{t+1,t+2,\dots,2t\})=\emptyset].
\label{merging of big blocks}
\end{align}
The probability on the right-hand side equals
\begin{align*}
\frac{1}{n^t}\sum_{k=1}^t \binom{n}{k}\cdot S(t, k)\cdot k! \cdot \left(1-\frac{k}{n}\right)^t\leq 
\sum_{k=1}^t n^{k-t}\cdot S(t, k) \cdot \left(1-\frac{k}{n}\right)^t .
\end{align*}
Let us denote $n^{k-t}\cdot S(t, k) \cdot \left(1-\frac{k}{n}\right)^t$ by $s_t(k)$. Then $s_t(t)=\left(1-\frac{t}{n}\right)^t$ and 
\[\frac{s_t(k-1)}{s_t(k)}=\frac{1}{n}\cdot \frac{S(t, k-1)}{S(t, k)}\cdot \left(1+\frac{1}{n-k} \right)^t\leq 
\frac{2}{n}\cdot \frac{t^2}{2(t-k+1)},\]
where the last inequality is due to Lemma \ref{quotient of stirling numbers}. This quantity is less than $1$ for all $k<t-(\log{n})^2$, hence the maximal value $\max_l \{s_t(l)\}$ is achieved for some $k>t-(\log{n})^2$. We thus have the following estimate:
\begin{align*}
\max_{1\leq l\leq t}\{s_t(l)\}&=s_t(k)=s_t(t)\cdot\frac{s_t(t-1)}{s_t(t)}\dots\frac{s_t(k)}{s_t(k+1)}\leq s_t(t)\cdot\prod_{l=k+1}^{t}\left(\frac{t^2}{n}\cdot\frac{\left(1+\frac{1}{n-l}\right)^t}{2(t-l+1)}\right) \\ & =\left(1-\frac{t}{n}\right)^t\cdot\left(1+\frac{t-k}{n-t}\right)^t\cdot (\log{n})^{2(t-k)}\cdot \frac{1}{2^{t-k}\cdot (t-k)!} \\ & \leq
\left(1-\frac{k}{n}\right)^t\cdot\left(\frac{e\cdot(\log{n})^2}{2(t-k)}\right)^{t-k}=(1+o(1))\cdot e^{-(\log{n})^2}\cdot\left(\frac{e\cdot(\log{n})^2}{2(t-k)}\right)^{t-k},
\end{align*} 
as $n\to\infty$. Define $x$ to be $\frac{t-k}{(\log{n})^2}$. Then $x\leq 1$ and we have 
\[
s_t(k)\leq 2\cdot e^{-(\log{n})^2}\cdot\left(\frac{e\cdot(\log{n})^2}{2(t-k)}\right)^{t-k} 
=2\cdot\left(\frac{\left(\frac{e}{2x}\right)^x}{e}\right)^{(\log{n})^2}\leq 2\cdot e^{-(\log{n})^2/2}
\]
because $\left(\tfrac{e}{2x}\right)^x\leq\sqrt{e}$ for $x\geq0$.
Now it follows immediately that the probability on the right-hand side of (\ref{merging of big blocks}) is at most $2t\cdot e^{-(\log{n})^2/2}$, hence 
\[\mathbb{P}[\sup\{p_{\phi},p'\}\not = p_{\max}]\leq n\cdot 2t\cdot  e^{-(\log{n})^2/2}< n^2\cdot e^{-(\log{n})^2/2}.\]
\vskip-2em
\end{proof}

We now prove that if we have substantially more than $\log{n}$ partitions then their supremum is likely to be equal to $p_{\max}$.
\begin{theorem}
Let $w:\mathbb{N}\rightarrow\mathbb{R}$ be a function such that $\lim_{n\rightarrow\infty}w(n)=\infty$. Let $t=t(n)=\log(n)+w(n)$ be an integer, and suppose maps $\phi_1, \phi_2,\dots,\phi_t:[n]\rightarrow [n]$ are chosen independently according to the uniform measure on the set of all maps. Then 
\[\lim_{n\rightarrow\infty}\mathbb{P}[\sup\{p_{\phi_1},\dots,p_{\phi_t}\}=p_{\max}]=1.\] 
\end{theorem}
\begin{proof}                                                      
We may assume that $w(n)<\log{n}$ since the more partitions we take the more likely their supremum is to be $p_{\max}$. Lemmas \ref{no blocks of size o(sqrt(n))} and \ref{no blocks of size less than log(n)sqrt(n)} show that the partition $p':=\sup\{p_{\phi_2},\dots,p_{\phi_t}\}$ has no blocks of size less than $\log{n}\cdot\sqrt{n}$ with probability $1-o(1)$. Hence, by Lemma \ref{large blocks merge} we have $p=\sup\{p',p_{\phi_1}\}=p_{\max}$ with probability $1-o(1)$.
\end{proof}

\section{The size of the largest block}\label{sec:largest}
In this section we study the typical picture for $p=\sup\{p_{\phi_1},p_{\phi_2},\dots,p_{\phi_t}\}$ when $t$ is fixed. For $t=3, 4$, $p$ is likely to have a block of size $\Omega(n)$, as shown in Theorems \ref{large block for t=3} and \ref{large block for t=4}, the former requiring much more subtle asymptotics for $S(n, k)$.  Theorem \ref{big block for big t} claims that for larger $t$, the partition $p$ is likely to have a block of size $n-\eps_t\cdot n$, where $\eps_t$ decays exponentially in $t$. We also show in Theorem \ref{no big blocks for t=2} that contrary to the case $t=3$, if we consider a supremum of two random partitions, it is likely to have no blocks of size $\Omega(n)$. 

For further results we need the notion of $k$--\emph{free} partition. For any $k<n$ define the set of partitions $E_k$ to be $\{p| \text{ there is no partition }p'\succeq p\text{ having a block of size }k\}$. We shall call partitions from the set $E_k$ $k$--\emph{free} partitions. We first formulate several simple properties of $k$--free partitions and prove them.
\begin{lemma}
Suppose a partition $p$ of $[n]$ is $k$--free for any $k\in[a, b]$ where $a<b$ are natural numbers, then $p$ has a block of size at least $b-a$.
\label{property of k--free partitions number 0}
\end{lemma}
\begin{proof}
Arguing by contradiction we suppose that $p$ has blocks of sizes $x_1\leq x_2\leq\dots x_r<b-a$. Let $s$ be the first index such that $x_1+\dots+x_s\geq a$, then $a\leq x_1+\dots+x_s=(x_1+\dots+x_{s-1})+x_s\leq a+(b-a)= b$, which is a contradiction as $p$ is not $x_1+\dots+x_s$--free.
\end{proof}

\begin{lemma}
Suppose a partition $p$ of $[n]$ is $k$--free for any $k\in[a, b]$ where $a, b$ are natural numbers satisfying $2\cdot a\leq b$ then $p$ has a block of size at least $b$.
\label{property of k--free partitions number 1}
\end{lemma}
\begin{proof}
By Lemma \ref{property of k--free partitions number 0} $p$ has a block of size at least $b-a\geq a$. Clearly $p$ cannot have a block of size $k\in [a, b]$ and thus $p$ has a block of size at least $b$.
\end{proof}
\begin{lemma}
Suppose a partition $p$ of $[n]$ has $h$ blocks of size $1$ and $b$--free for some $b>h$, then $p$ has a block of size at least $h$.
\label{property of k--free partitions number 2}
\end{lemma}
\begin{proof}
Consider a partition $p'$ of a set with $n-h$ elements obtained by removing all blocks of size $1$ from $p$. It is easy to see that $p'$ is $k$--free for any $k\in[b-h, b]$. Indeed, if a union of some blocks in $p'$ had size $b-k$ for some $k\leq h$ then adding $k$ blocks of size $1$ we would deduce that a union of some blocks in $p$ has size $b$ which is impossible. By Lemma \ref{property of k--free partitions number 0} $p'$ has a block of size at least $h$. Hence, $p$ has a block of size at least $h$. 
\end{proof}

\begin{lemma}
Suppose a partition $p$ of $[n]$ is $k$--free for any $k\in[a, b]$ where $a, b$ are natural numbers satisfying $2\cdot a\leq b$. Then the size of the union of all blocks in $p$ which have size at least $a$ is bounded below by $n-a$.
\label{property of k--free partitions number 3}
\end{lemma}
\begin{proof}
We argue by contradiction. We shall call blocks of size at least $a$ \emph{big} and all other blocks \emph{small}. Suppose the union of all big blocks has size smaller than $n-a$, then the union of small blocks has size greater than $a$. Let $x_1,x_2,\dots ,x_r$ be the sizes of small blocks, thus $x_1+x_2+\dots+x_r>a$. Consider the smallest index $i$ such that $x_1+\dots+x_i\geq a$, then we have $a\leq x_1+x_2+\dots+x_i= (x_1+\dots x_{i-1})+x_i\leq a+a\leq b$ which is a contradiction as $p$ cannot be $(x_1+\dots+x_i)$--free.
\end{proof} 
\begin{theorem}
\label{large block for t=4}
Suppose four maps $\phi_1,\phi_2,\phi_3,\phi_4:[n]\rightarrow [n]$ are chosen independently according to the uniform measure. Let $p=\sup\{p_{\phi_1},p_{\phi_2},p_{\phi_3},p_{\phi_4}\}$, then
\[\lim_{n\rightarrow\infty}\mathbb{P}\left[p \text{ has a block of size at least }\frac{n}{3}\right]=1.\]
\end{theorem}
\begin{proof}
Let $c\in [\frac{1}{11}, \frac{1}{3}]$ be a constant and let $k=c\cdot n$. We shall see that $p$ is exponentially unlikely to be in $\Pi_n\setminus E_k$, and thus with high probability it lies in $E_k$. Similarly to the proof of Lemma \ref{no blocks of size o(sqrt(n))} we have
\begin{align*}
\mathbb{P}[p\not\in E_k] & \leq \binom{n}{k}\left(\mathbb{P}[\phi(a)\not = \phi(b) \text{ for any } a\leq k < b]\right)^4 \\ & =
\binom{n}{k}\left(\sum_{l=1}^k \frac{1}{n^n}\cdot\binom{n}{l}\cdot l!\cdot S(k,l)\cdot(n-l)^{n-k}\right)^4 \\ & \leq
\binom{n}{k}\left(\sum_{l=1}^k n^{l-k}\cdot S(k,l)\cdot\left(1-\frac{l}{n}\right)^{n-k}\right)^4.
\end{align*}

We again denote $n^{l-k}\cdot S(k,l)\cdot\left(1-\frac{l}{n}\right)^{n-k}$ by $f_k(l)$. We want to bound $\max_l\{f_k(l)\}$. Note that due to Lemma \ref{quotient of stirling numbers} we have  
\[\frac{f_k(l-1)}{f_k(l)}=\frac{1}{n}\cdot\frac{S(k,l-1)}{S(k,l)}\cdot\left(1+\frac{1}{n-l}\right)^{n-k}\leq \frac{e\cdot l^2}{2\cdot n\cdot (k-l+1)},\] 
thus the maximum of $f_k(l)$ over $l$ for fixed $k,n$ must be achieved for some $l\geq k/2.$  We now bound $f_k(l)$.
\begin{align*}
f_k(l)=f_k(k)\cdot\prod_{r=l+1}^k \left(\frac{f_k(r-1)}{f_k(r)}\right)\leq f_k(k)\cdot\prod_{r=l+1}^k \left(\frac{e\cdot r^2}{2\cdot n \cdot (k-r+1)}\right).
\end{align*}
Note that the factors which correspond to $r<k-x\cdot n$ are smaller than $1$, where $x$ is the smallest solution of the equation
\begin{align}
2\cdot x= e\cdot (c-x)^2.
\label{equation for x and c}
\end{align}

We can then disregard all these factors to deduce that 
\begin{align*}
\sum_{l=1}^k f_k(l)&\leq k\cdot\max_l\{f_k(l)\} \\ &\leq 
k\cdot f_k(k)\cdot\prod_{r=(c-x)\cdot n}^{c\cdot n}\left(\frac{e\cdot r^2}{2\cdot n \cdot (k-r+1)}\right) \\ &=
k\cdot \left(1-\frac{k}{n}\right)^{n-k}\cdot \left(\frac{e}{2\cdot n}\right)^{xn}\cdot\left(\frac{(cn)!}{((c-x)\cdot n)!}\right)^2\cdot\frac{1}{(xn)!}.
\end{align*}
We now take $\log$ and divide through by $n$ to obtain
\begin{align*}
\frac{\log{(\max_l\{f_k(l)\})}}{n}&\leq 
(1-c)\cdot\log{(1-c)}+x\cdot(1-\log{2})-x\cdot\log{n}+2\cdot c\log{(cn)}-2\cdot c\\&-
2\cdot (c-x)\log{((c-x)n)}+2\cdot (c-x)-x\cdot \log{(xn)}+x+o(1).
\end{align*}
We see that all summands of order $\log{n}$ magically cancel out and we obtain
\begin{align*}
\frac{\log{\left(\sum_{l=1}^k f_k(l)\right)}}{n}&\leq 
(1-c)\cdot\log{(1-c)}+x\cdot(2-\log{2})+2\cdot c\log{c}-2\cdot c\\&-
2\cdot (c-x)\log{(c-x)}+2\cdot (c-x)-x\cdot \log{x}+o(1)\\&=
(1-c)\cdot\log{(1-c)}-x\cdot\log{2}+2\cdot c\log{c}\\&-
2\cdot (c-x)\log{(c-x)}-x\cdot \log{x}+o(1).
\end{align*}
We denote the right-hand side by $\mu(c)$. Note that it indeed depends on $c$ only, as $x$ can be expressed in terms of $c$ using (\ref{equation for x and c}). We have now the following estimate for $\mathbb{P}[p\not\in E_k]$:
\begin{align*}
\mathbb{P}[p\not\in E_k]&\leq \binom{n}{k}\cdot\left(\sum_{l=1}^k n^{l-k}\cdot S(k,l)\cdot\left(1-\frac{l}{n}\right)^{n-k}\right)^4 \\&\leq
e^{n\cdot (H(c)+4\mu(c)+o(1))},
\end{align*}
where we use the standard notation for the entropy function:
\begin{equation}
H(c)=-c\cdot\log{c}-(1-c)\cdot\log{(1-c)}.
\label{entropy function}
\end{equation}
Thus $\mathbb{P}[p\not\in E_k]$ decays exponentially whenever $\lambda(c):=H(c)+4\cdot \mu(c)<0$ which turns out to be the case for $c\in [0.087412, 0.340034]$. Moreover, since $\lambda(c)$ is continuous, there exists $\eps>0$ such that $\lambda(c)<-\eps$ for all $c\in [1/11, 1/3]$. Hence, the union bound gives us 
\[\mathbb{P}\left[p\not\in  E_k\text{ for some } k\in\left[\frac{n}{11}, \frac{n}{3}\right]\right]\leq n\cdot e^{-\eps\cdot n}. \]
It follows from Lemma \ref{property of k--free partitions number 1} that $p$ has a block of size at least $\frac{n}{3}$ with probability at least $1-n\cdot e^{-\eps\cdot n}$.
\end{proof}

\begin{theorem}
\label{big block for big t}
For any $\eps\in(0,\log(e/2))$ there exists a constant $C>0$ such that for any fixed $t>C$ the following holds. Suppose maps $\phi_1,\phi_2,\dots,\phi_t:[n]\rightarrow [n]$ are chosen independently according to the uniform measure. Let $p=\sup\{p_{\phi_1},p_{\phi_2},\dots,p_{\phi_t}\}$ and denote the size of the largest block in $p$ by $L$, then
\[\lim_{n\rightarrow\infty}\mathbb{P}\left[1-\tfrac12e^{-t+3-\eps}\leq\frac{L}{n}\leq 1-e^{-t-\eps}\right]=1.\]
\end{theorem}
\begin{proof}
We first prove that $\lim_{n\rightarrow\infty}\mathbb{P}\left[L<n-n\cdot e^{3-t}\right]=0$. The proof is similar to the one of Theorem \ref{large block for t=4}, though this time the calculations are easier. We consider the partition $p'=\sup\{p_{\phi_2},p_{\phi_3},\dots p_{\phi_t}\}$. Take $k=n\cdot e^{c-t}$ for some $c\in[2+\delta,3]$, with $\delta>0$ small and fixed, we again have 
\begin{align*}
\mathbb{P}[p'\not\in E_k] & \leq \binom{n}{k}\left(\sum_{l=1}^k n^{l-k}\cdot S(k,l)\cdot\left(1-\frac{l}{n}\right)^{n-k}\right)^{t-1}.
\end{align*}
This time we write the following rougher bound, for $1\leq l\leq k$
\[\frac{f_k(l-1)}{f_k(l)}=\frac{1}{n}\cdot\frac{S(k,l-1)}{S(k,l)}\cdot\left(1+\frac{1}{n-l}\right)^{n-k}\leq \frac{2\cdot k^2}{n\cdot (k-l+1)},\]
and this quantity is smaller than $1$ for $l<k-\frac{2\cdot k^2}{n}$.
Hence we obtain the following bound:
\begin{align*}
\sum_{l=1}^k f_k(l)&\leq k\cdot\max_l\{f_k(l)\} \\ &\leq 
k\cdot f_k(k)\cdot\prod_{r=k-k^2/n}^{k}\left(\frac{2 k^2}{n \cdot (k-r+1)}\right) \\ &\leq
k\cdot \left(1-\frac{k}{n}\right)^{n-k}\cdot 2^{k^2/n}\cdot e^{k^2/n} \\ &\leq
e^{-k+5k^2/n}.
\end{align*}
Let us denote $\frac{k}{n}=e^{c-t}$ by $\alpha$. Recalling the notation $H(\alpha)$ from (\ref{entropy function}) we have the following inequality:
\begin{align*}
\mathbb{P}[p'\not\in E_k] & \leq \binom{n}{k}\left(\sum_{l=1}^k f_k(l)\right)^{t-1}\leq e^{n\cdot H(\alpha)}\cdot e^{-(t-1)\cdot n\cdot(\alpha-5\cdot\alpha^2)}.
\end{align*} 
In order to conclude the proof we use an estimate $H(\alpha)\leq -\alpha\log{\alpha}+\alpha$ which allows us to write
\begin{align*}
H(\alpha)-(t-1)\cdot (\alpha-5\cdot \alpha^2)&\leq (-\alpha\cdot\log{\alpha}+\alpha-(t-1)\cdot (\alpha-5\cdot \alpha^2))\\&=
((2-c)+5\cdot (t-1)\cdot\alpha)\cdot\alpha\leq \frac{-\delta\cdot\alpha}{2}.
\end{align*}
The last inequality is valid for sufficiently large $t$, since $2-c<-\delta$ and $\alpha=e^{c-t}$, so $t\cdot\alpha$ vanishes as $t\to\infty$. Consequently, the union bound gives us an estimate for the probability that $p'\in E_k$ for some $k\in[e^{2+\delta-t}\cdot n, e^{3-t}\cdot n]$:
\[\mathbb{P}[\text{there exists }k\in[e^{2+\delta-t}\cdot n, e^{3-t}\cdot n] \text{ such that }p'\not\in E_k]\leq e^{-\delta\alpha\cdot n/4}.\]
We denote $\cap_{k\in [e^{2+\delta-t}\cdot n, e^{3-t}\cdot n]} E_k$ by $E$. The above statement claims that $\mathbb{P}[p'\not\in E]\leq e^{-\delta\alpha\cdot n/4}$. For $\delta<1-\log{2}$, say for $\delta=1-\log{2}-\eps$ with $\eps$ mentioned in the claim, by Lemma \ref{property of k--free partitions number 3} we know that for any $p'\in E$ the union of all blocks of size at least $c_1:=e^{2+\delta-t}=\tfrac12e^{3-t-\eps}$ in $p'$ has size at least $(1-c_1)\cdot n$. Finally the argument presented in the proof of Lemma \ref{large blocks merge} shows that in $p=\sup\{p',p_1\}$ all these blocks are merged with probability $1-o(1)$ and thus $p$ has a block of size at least $(1-c_1)\cdot n$  with probability tending to 1.

 We now prove that  $\lim_{n\rightarrow\infty}\mathbb{P}\left[L>n-n\cdot e^{-t-\eps}\right]=0$. The argument is very similar to the one presented in the proof of Theorem \ref{supremum is not p_max for t<log(n)}. Let $M$ again be the number of one-element blocks in $p$. As $n\to\infty$ with $t$ fixed by Lemma~\ref{moments of M} we have $\mathbb{E}[M]=e^{-t}\cdot n\cdot (1+O(1/n))$ and $\var(M)=O(n)$. 
Trivially, the largest block has size $n-M$ at most, thus by Chebyshev's inequality
\begin{align*}
\mathbb{P}\left[L>n-n\cdot e^{-t-\eps}\right]\leq\mathbb{P}\left[M<n\cdot e^{-t-\eps}\right]\leq\frac{\var(M)}{\left(\mathbb{E}[M]-n\cdot e^{-t-\eps}\right)^2}=O\left(\frac{1}{n}\right).
\end{align*}
\vskip-1em
\end{proof}

In order to prove Theorem \ref{large block for t=3} we need the following lemma which claims that the supremum of three random partitions is likely to have $\Omega(n)$ one-element blocks.
\begin{lemma}
\label{many one-element blocks for t=3}
Suppose three maps $\phi_1,\phi_2,\phi_3:[n]\rightarrow [n]$ are chosen independently according to the uniform measure. Let $p=\sup\{p_{\phi_1},p_{\phi_2},p_{\phi_3}\}$, then for any constant $c\in(0,e^{-3})$ 
\[\lim_{n\rightarrow\infty}\mathbb{P}\left[p \text{ has at least }c\cdot n \text{ one-element blocks}\right]=1.\]
\end{lemma}
\begin{proof}
The proof uses the same technique as presented in the proof of Theorem \ref{supremum is not p_max for t<log(n)}. Let $M$ denote the number of one-element blocks in $p$. By Lemma~\ref{moments of M} we have, as $n\to\infty$, $\mathbb{E}[M]=e^{-3}\cdot n+O(1)$, $\var(M)=O(n)$. So, by Chebyshev's inequality, the probability that $M$ is less than $c\cdot n$ is
\[\mathbb{P}[M<c\cdot n]\leq\frac{\var(M)}{(\mathbb{E}[M]-c\cdot n)^2}=O\left(\frac{1}{n}\right).\]
\vskip-2em
\end{proof}

\begin{theorem}
\label{large block for t=3}
There exists a constant $c>0$ such that the following holds. Suppose three maps $\phi_1,\phi_2,\phi_3:[n]\rightarrow [n]$ are chosen independently according to the uniform measure. Let $p=\sup\{p_{\phi_1},p_{\phi_2},p_{\phi_3}\}$, then
\[\lim_{n\rightarrow\infty}\mathbb{P}\left[p \text{ has a block of size at least }c\cdot n\right]=1.\]
\end{theorem}
\begin{proof}
In the light of Lemma \ref{many one-element blocks for t=3} and Lemma \ref{property of k--free partitions number 2} it is sufficient to show that $\mathbb{P}[p\in E_{n/2}]$ tends to $1$ as $n$ grows to infinity. Similarly to the proof of Theorem \ref{large block for t=4} we have 
\begin{align*}
\mathbb{P}[p\not\in E_{n/2}] & \leq \binom{n}{n/2}\left(\mathbb{P}[\phi(a)\not = \phi(b) \text{ for any } a\leq n/2 < b]\right)^3 \\ & =
\binom{n}{n/2}\left(\sum_{l=1}^{n/2} \frac{1}{n^n}\cdot\binom{n}{l}\cdot l!\cdot S(n/2,l)\cdot(n-l)^{n/2}\right)^3 \\ & \leq
n^2\cdot\left(\sum_{l=1}^{n/2}\left(2^{n/3}\cdot \frac{n!}{n^{n/2}\cdot(n-l)!}\cdot S(n/2, l)\cdot\left(1-\frac{l}{n}\right)^{n/2}\right)^3\right).
\end{align*}
Here we used that $(x_1+\dots+x_m)^3\leq m^2\cdot(x_1^3+\dots+x_m^3)$. The idea is now to prove that $f(l)=2^{n/3}\cdot \frac{n!}{n^{n/2}\cdot(n-l)!}\cdot S(n/2, l)\cdot\left(1-\frac{l}{n}\right)^{n/2}$ decays exponentially in $n$ uniformly in $l$. For that write $l=c\cdot n/2$ with $c\in (0,1)$ and use Lemma \ref{tight asymptotic for stirling numbers} together with the asymptotic formula $n!=n^n\cdot e^{-n+o(n)}$ to obtain
\begin{align*}
\frac{\log{f(l)}}{n}&=\frac{1}{n}\cdot\log{\left(2^{n/3}\cdot  \frac{n!}{n^{n/2}\cdot(n-l)!}\cdot S(n/2, l)\cdot\left(1-\frac{l}{n}\right)^{n/2}\right)} \\&=
\frac{1}{n}\cdot\log{\left(2^{n/3}\cdot\left(1-\frac{c}{2}\right)^{-n\cdot(1-c/2)}\cdot 2^{-n/2+l}\cdot e^{-cn/2}\cdot e^{g(c)\cdot n/2+o(n)}\cdot\left(1-\frac{c}{2}\right)^{n/2}\right)}\\&=
\log{2}\cdot\left(\frac{c}{2}-\frac{1}{6}\right)-\left(1-\frac{c}{2}\right)\log{\left(1-\frac{c}{2}\right)}-\frac{c}{2}+\frac{g(c)}{2}+\frac{\log{\left(1-\frac{c}{2}\right)}}{2}+o(1)\\&=\mu(c)+o(1).
\end{align*}
Here $g(c)=c+\log{\gamma}+(\gamma-c)\cdot\log{(\gamma-c)}-\gamma\cdot\log{\gamma}$ and $\gamma$ is given by $\gamma\cdot(1-e^{-1/\gamma})=c$.
It remains to note that $\mu(c)<-\eps$ for all $c\in(0, 1)$, with some fixed $\eps>0$.
\end{proof}
\begin{rem}
It turns out that $\max_{c}\mu(c)$ is very close to $0$, namely, $0>\max_{c}\mu(c)>-\frac{1}{500}$. 
\end{rem}

\begin{theorem}
\label{no big blocks for t=2}
For any $\eps>0$ the following holds. Suppose two maps $\phi_1, \phi_2:[n]\rightarrow [n]$ are chosen independently according to the uniform measure on the set of maps. Let $p=\sup\{p_{\phi_1},p_{\phi_2}\}$, then 
\[\lim_{n\rightarrow\infty}\mathbb{P}[p \text{ has a block of size at least } n^{\frac{3}{4}+\eps}]=0.\] 
\end{theorem}
\begin{proof}
Consider a graph $G$ on $n$ vertices with edges of two types. We draw an edge of the first type between two vertices $i$ and $j$ if $\phi_1(i)=\phi_1(j)$ and an edge of the second type if $\phi_2(i)=\phi_2(j)$. Note that it is possible that there are edges of both types between $i$ and $j$. It is clear from the construction that blocks of $p=\sup\{p_{\phi_1},p_{\phi_2}\}$ are connected components of the graph $G$. Evidently, as edges of each type form disjoint cliques, if there exists a path in $G$ between $i$ and $j$ then there also exists a simple path between $i$ and $j$ in which types of edges alternate.

Let us now fix two vertices $i$ and $j$ in $G$ and estimate the probability that there exists such an alternating path. For a fixed $k$, the probability of having an alternating path of length $k$ between $i$ and $j$ is bounded above by $2\cdot\frac{(n-2)!}{(n-k-1)!}\cdot n^{-k}$. Indeed, $\frac{(n-2)!}{(n-k-1)!}$ is the number of sequences of vertices $\gamma=\{\gamma_0,\gamma_1,\dots,\gamma_k\}$ between $i=\gamma_0$ and $j=\gamma_k$, and for a fixed sequence $\gamma$ the probability that it forms an alternating path in $G$ is $2\cdot(n^{-1})^k$, simply because the probability that two given vertices are connected by an edge of a given type is $\frac{1}{n}$. We thus obtain the following bound:
\begin{align*}
\mathbb{P}&[\text{there exists a path between }i \text{ and }j \text{ in }G]\leq 2\cdot \sum_{k=1}^{n-1}\frac{(n-2)!}{(n-k-1)!}\cdot n^{-k} \\ &= \frac{2}{n}\cdot\left(1+\left(1-\frac{2}{n}\right)+\left(1-\frac{2}{n}\right)\cdot\left(1-\frac{3}{n}\right)+\dots\right) \\& \leq
\frac{2}{n}\cdot\left(n^{1/2}+n^{1/2} + n^{1/2}\cdot \left(1-\frac{n^{1/2}}{n}\right)^{n^{1/2}}+n^{1/2}\cdot\left(1-\frac{n^{1/2}}{n}\right)^{2\cdot n^{1/2}}+\dots \right) \\ &= 
\frac{2\cdot n^{1/2}}{n}\cdot \left(1+1+e^{-1}+e^{-2}+\dots\right) = O\left(n^{-1/2}\right).
\end{align*}
It follows then immediately that the expected number of pairs $\{i,j\}$ such that $i$ and $j$ are in the same block in $p$ is at most $O(n^{\frac{3}{2}})$. Hence,
\[\mathbb{P}[p \text{ has a block of size at least } n^{\frac{3}{4}+\eps}]\leq \frac{O\left(n^{\frac{3}{2}}\right)}{\left(n^{\frac{3}{4}+\eps}\right)^2}=O(n^{-2\eps}).\]
\vskip-2em
\end{proof}

\end{document}